\documentclass[a4 paper,11pt]{article}
\usepackage{stmaryrd}
\usepackage{amsfonts}
\usepackage{bbm}
\usepackage{amscd}
\usepackage{mathrsfs}
\usepackage{latexsym,amssymb,amsmath,amscd,amscd,amsthm,amsxtra,xypic}
\usepackage[dvips]{graphicx}
\usepackage[utf8]{inputenc}
\usepackage[T1]{fontenc}
\usepackage{enumerate}
\usepackage{amssymb}
\usepackage[all]{xy}
\usepackage{ifpdf}
\ifpdf
\usepackage[colorlinks=true,linkcolor=blue,final,backref=page,hyperindex,citecolor=red]{hyperref}
\else
\usepackage[colorlinks,final,backref=page,hyperindex,hypertex]{hyperref}
\fi

\usepackage{nicefrac,mathtools,enumitem}
\usepackage{microtype}
\usepackage{amsfonts}
\usepackage{amssymb}
\usepackage{mathrsfs}
\usepackage{amsmath}
\usepackage{amsthm}
\usepackage{enumerate}
\usepackage{indentfirst}
\usepackage{amsfonts}
\usepackage{amssymb}
\usepackage{mathrsfs}
\usepackage{amsmath}
\usepackage{amsthm}
\usepackage{enumerate}
\usepackage{cite}
\usepackage{mathrsfs}
\usepackage{geometry}
\usepackage{multirow}
\allowdisplaybreaks

\usepackage{dsfont}

    \newtheorem{thm}{Theorem}[section]
    \newtheorem{prop}[thm]{Proposition}
    \newtheorem{defnprop}[thm]{Definition-Proposition}  
   \newtheorem{lem}[thm]{Lemma}
    \newtheorem{coro}[thm]{Corollary}

\theoremstyle{defn}
    \newtheorem{defn}[thm]{Definition}

    \newtheorem{ex}[thm]{Example}
    \newtheorem{exs}[thm]{Examples}

\theoremstyle{remark}
    \newtheorem{rmk}[thm]{Remark}

\newcommand{\dd}{\diamond}

\newcommand{\Courant}[1]{\left\llbracket  #1\right\rrbracket }
\usepackage{hyperref}
\hypersetup{
	colorlinks,
	citecolor=blue,
	filecolor=black,
	linkcolor=blue,
	urlcolor=black
}

\def\ns{\curlyvee}
\begin{document}
%%%%%%%%%%%%%%%%%%%%%%%%%%%%%%%%%%%%%%%%%%%%%%%%%%%%%%%
\title{\sf Rota-Baxter type  operators on trusses and derived structures}
\author{\bf  T. Chtioui$^{1}$, M. Elhamdadi$^{2}$,  S. Mabrouk$^{3}$ and A. Makhlouf$^{4}$} 
\date{ 
$^{1}$ Mathematics and Applications Laboratory LR17ES11, Faculty of Sciences, Gabes University, Tunisia.\\
$^{2}$
Department of Mathematics,
University of South Florida, Tampa, FL 33620, U.S.A.
\\%benhassine.abdelkader@yahoo.fr; chtioui.taoufik@yahoo.fr \authorcr
$^{3}$ Faculty of Sciences, University of Gafsa, BP 2100, Gafsa, Tunisia.\\
$^{4}$ Universit\'e de Haute Alsace~ IRIMAS - D\'epartement de Math\'ematiques,  Mulhouse, France}
\maketitle

%\author{Ramkrishna Mandal}
%\address{Department of Mathematics, Indian Institute of Technology, Kharagpur 721302, West Bengal, India.}
%\email{ramkrishnamandal430@gmail.com}

\begin{abstract}
The aim of this paper is to introduce and study the concepts of the Rota-Baxter operator and Reynolds operator within the framework of trusses. Moreover, we introduce and discuss dendriform trusses, tridendriform trusses, and NS-trusses as fundamental algebraic structures underlying these classes of operators. Furthermore, we consider the notions of Nijenhuis operator and averaging operator to trusses, exploring their properties and applications to uncover new algebraic structures.
\end{abstract}

\maketitle

%\curraddr{}
%\email{}

%\subjclass[2010]{}
%\keywords{}

\medskip

\medskip

\medskip



\thispagestyle{empty}
\numberwithin{equation}{section}

\tableofcontents

%\vspace{0.2cm}

\medskip

%%%%%%%%%%%%%%%%%%%%%%%%%%%%%%%%%%%%%%%%%%
\section{Introduction}\label{sec1}
%%%%%%%%%%%%%%%%%%%%%%%%%%%%%%%%%%%%%%%%%%
%================================
\textbf{Heaps and trusses}.
A truss \cite{ Brz:par} is an algebraic structure comprising of a set equipped with a ternary operation that forms an abelian heap \cite{Pru:the, Bae:ein}, along with an associative binary operation that distributes over the ternary operation. From the perspective of universal algebra, its composition involves fewer operations, making it simpler than a ring (which requires two binary operations, one unary operation, and one nullary operation) or a (two-sided) brace \cite{Rum:bra, CedJes:bra}. A brace, defined as a set with two intertwined group structures, involves six operations and has recently gained significant attention due to its connection with the set-theoretic Yang-Baxter equation.

A truss can be transformed into a ring by introducing a specific nullary operation or an element with special properties. If the binary operation in a truss is a group operation, it can be associated with a (two-sided) brace. Conversely, every ring can naturally be viewed as a truss by defining the ternary heap operation as 
$[a,b,c]=a-b+c$, derived from the abelian group operation. Similarly, every brace can also be interpreted as a truss. This demonstrates that trusses are not only structurally simpler but also more general than rings. However, their definition relies on a ternary operation, which is less intuitive and familiar compared to binary operations, presenting a classic trade-off between generality and ease of understanding.

%==========================================
\textbf{Rota-Baxter type operators}.
The study of operators on various algebraic structures is a well-established and classical area of research in algebra. Among the most extensively studied operators are endomorphisms and automorphisms, which can be defined on the generating set of an algebraic system and then extended to all its elements. However, there are other types of operators whose construction is more complex. Rota-Baxter operators appeared first in Baxter's seminal work \cite{Bax} in 1960 in fluctuation theory. Since then, the theory of Rota-Baxter operators has been developed extensively across various fields of mathematics. 
Given an algebra $A$, a linear map $R : A \rightarrow A$ is a Rota-Baxter operator of weight $\lambda$, if for all $x,y\in A$, 
 \begin{equation}
 R(x)R(y)=R(xR(y))+R(R(x)y)+\lambda R(xy).
 \end{equation}
These operators are significant due to their connections with key mathematical concepts, including the Yang-Baxter equation \cite{MR0674005, MR0725413,MR2568415}, Dendriform algebras \cite{MR3021790} and renormalization of quantum field theory \cite{CK}. A comprehensive historical overview and their connections to other mathematical ideas can be found in  \cite{G}.
Rota-Baxter operators on groups were introduced in \cite{GLS}. 
Let $G$ be a group. A map $R\colon G\to G$ is called a  Rota-Baxter operator of weight~1  (resp. weight~$-1$) if, for all $g,h\in G$,
\begin{equation}
R(g)R(h) = R( g R(g) h R(g)^{-1} ), \quad (\text{resp. } R(g) R(h) = R( R(g) h R(g)^{-1} g ).
\end{equation}
The properties of Rota-Baxter groups have been actively studied in \cite{BG}.  Connections  between Rota-Baxter groups,  Yang-Baxter equation and skew braces were established in \cite{BG-1}. 

Other notable classes of operators include  Reynolds operators and averaging operators, which have emerged independently in different branches of mathematics.
Let $E$ be a locally compact Hausdorff space and  $C_0(E)$ denote the Banach algebra of all continuous functions on $E$ vanishing at $\infty$. A Reynolds operator $P:  C_0(E) \to C_0(E)$ (see \cite{R}) is a linear map  that satisfies the Reynolds identity:
\[
P(x) P(y) = P \left( P(x) y + x P(y) - P(x) P(y)\right), \qquad\forall x, y \in C_0(E).
\]
Obviously, a projection $P$ is a Reynolds operator if $P$ is averaging, i.e.
\[
P(x) P(y) = P \left( P(x) y\right), \qquad\forall x, y \in C_0(E).
\]
The Reynolds operator was originally introduced in the context of turbulence theory. Averaging operators have been studied extensively by Birkhoff \cite{Bir1, Bir2}, Kelley \cite{K} and  Rota \cite{R1, R2}, see also  \cite{Safa, HP, FM, KKL, PB,  BCEM,S}, as well as references therein.

Inspired by the extensive body of work on operators in algebraic systems, particularly the study of Rota-Baxter operators, Nijenhuis operators, Reynolds operators and  averaging operators, this paper extends these concepts to the setting of trusses which provide a natural yet under-explored framework for such operators.

In this work, we introduce Rota-Baxter operators, Reynolds operators, Nijenhuis operators and averaging operators on trusses, systematically constructing all derived structures and investigating their properties. We explore the interplay between these operators and the underlying truss structure, highlighting how the ternary operation influences their behavior and properties. Additionally, we establish connections between these operators and other algebraic systems, such as rings, tridendriform trusses, ditrusses and NS-trusses, demonstrating how trusses serve as a unifying framework that generalizes and simplifies these structures.

By developing this theory, we aim to provide a deeper understanding of the algebraic and structural properties of trusses, while also uncovering new connections to well-studied operators in algebra. This work not only enriches the theory of trusses but also opens up new directions for research in universal algebra, operator theory, and their applications.

%=========================================
\textbf{Layout of the paper}.
In Section \ref{Sec2}, we provide a brief review of fundamental concepts related to heaps and trusses, establishing the necessary groundwork for the subsequent discussions. Section \ref{Sec3} focuses on the introduction of Rota-Baxter operators on trusses, where we present key characterizations and properties of these operators. In addition, we explore various algebraic structures that naturally arise from their application. In Section \ref{Sec4}, we introduce the notion of Reynolds operator on trusses, investigate their role in generating associated algebraic systems, and highlight their significance in structural analysis. Finally, in the last section, we consider  averaging operators and the interconnected algebraic structures they induce, offering further insights into their theoretical and practical implications.

\section{Preliminaries on heaps and trusses}\label{Sec2}

In this section, we review the basics  of heaps and trusses, along with some necessary notations. Further details can be found in, e.g., \cite{Brz:par, Bae:ein, Pru:the}. We give a classification of truss on $\mathbb Z_2$ and some construction results on trusses which will be useful in the sequel.

\begin{defn}[Abelian Heap]
An \emph{abelian heap} is a set $H$ equipped with a ternary operation $[\cdot,\cdot,\cdot]$ satisfying, for all $a,b,c,x,y \in H$, the following identities:
\begin{equation}\label{assoc}
  [[a,b,c],x,y]=[a,b,[c,x,y]],
\end{equation}
\begin{equation}\label{Malcev}
[a,a,b] = [b,a,a] = b,
\end{equation}
\begin{equation}\label{comm}
[a,b,c] = [c,b,a].
\end{equation}

\end{defn}
Equation \eqref{assoc} represents the \emph{associativity law} for heaps, while \eqref{Malcev} are known as the \emph{Mal'cev identities}. Equation \eqref{comm} expresses the \emph{commutativity law} for heaps.

%Due to these axioms, the way %brackets are distributed in multiple applications of the ternary operation does not affect the result. Therefore, we write:
%\[
%[a_1,a_2,\dots, a_{2n+1}]
%\]
%to denote the element of $H$ obtained through any possible application of the ternary operation to the (always odd-length) tuple $(a_1,a_2,\dots, a_{2n+1}) \in H^{2n+1}$. 

%Moreover, the Mal’cev identities and the commutativity law lead to the following cancellation and rearrangement rules for all $a_1,a_2,\dots, a_{2n+1} \in H$:
%\begin{equation}\label{cancel}
%[a_1,\dots, a_{i-1},a_i,a_i,a_{i+1},\dots, a_{2n}] = [a_1,\dots, a_{i-1},a_{i+1},\dots, a_{2n}],
%\end{equation}
%\begin{equation}\label{symm}
%[a_1,a_2,\dots, a_{2n+1}] = [a_{\varpi(1)},a_{\varsigma(2)},a_{\varpi(3)},\dots, a_{\varsigma(2n)},a_{\varpi(2n+1)}],
%\end{equation}

%for any permutation $\varpi$ of the set $\{1,3,5,\dots, 2n+1\}$ and any permutation $\varsigma$ of $\{2,4,\dots, 2n\}$.

\begin{defnprop}[Retract of a Heap]
For any $e \in H$, the set $H$ equipped with the binary operation
\[
+_e = [-,e,-]
\]
forms an abelian group, known as the \emph{retract} of $H$. The chosen element $e$ serves as the identity element of this group, and the inverse $-_e a$ of an element $a \in H$ is given by:
\[
-_e a = [e,a,e].
\]
\end{defnprop}
We denote this unique (up to isomorphism) group by $\mathcal{G}(H;e)$. Conversely, for any abelian group $G$, the operation
\[
[a,b,c] = a - b + c
\]
defines a heap structure on $G$, which we denote by $\mathcal{H}(G)$. 

A \emph{homomorphism of heaps} is a map that preserves the heap operation. In particular, any group homomorphism is also a heap homomorphism for the corresponding heap. Thus, the assignment $\mathcal{H}: G \to \mathcal{H}(G)$ defines a functor from the category of abelian groups to the category of abelian heaps.

\begin{exs}$ $
    \begin{enumerate}
    \item 
    The empty set, denoted by \(\emptyset\), forms a trivial heap when equipped with the ternary operation:
    \[
    \emptyset \times \emptyset \times \emptyset \rightarrow \emptyset.
    \]
    This operation is vacuously defined since there are no elements to operate on, making \(\emptyset\) the simplest example of a heap.
    \item 
    Any singleton set, say \(\{\ast\}\), can be endowed with a heap structure by defining the ternary operation as:
    \[
    [\ast, \ast, \ast] = \ast.
    \]
    This heap serves as the terminal object in the category of heaps, denoted by \(\mathbf{Hp}\), meaning there is a unique morphism from any heap to \(\{\ast\}\).
    \item 
    Consider the set \(H = 2\mathbb{Z} + 1\) consisting of all odd integers. This set forms an abelian heap under the ternary operation:
    \[
    [2n + 1, 2m + 1, 2p + 1] = 2(n - m + p) + 1,
    \]
    for all integers \(n, m, p \in \mathbb{Z}\). The abelian property ensures that the operation is commutative in a suitable sense.
    Furthermore, if we fix the element \(1 \in H\), the set \(H\) can be retracted to a group structure. The group operation is defined as:
    \[
    (2m + 1) \cdot (2n + 1) = 2(m + n) + 1,
    \]
    for all \(m, n \in \mathbb{Z}\), with \(1\) acting as the neutral element. This group is isomorphic to the additive group of integers \(\mathbb{Z}\) via the map:
    \[
    2m + 1 \mapsto m.
    \]
    This isomorphism highlights the deep connection between the heap structure and the underlying group structure.
\end{enumerate}

\end{exs}

\begin{defn}[Truss]
A \emph{truss} is  an abelian heap  $(T,[\cdot,\cdot,\cdot])$  equipped with  a binary operation,  denoted by $\cdot$ or just by  juxtaposition and called \emph{multiplication}, such that 
 $(T,\cdot)$ is a semigroup and
 for all $a,b,c,d\in T$, we have 
    \begin{equation}\label{dist}
        a[b,c,d] = [ab, ac, ad] \quad \text{and} \quad [a,b,c]d = [ad, bd, cd].
    \end{equation}
\end{defn}
A \emph{morphism of trusses} is a map that is a homomorphism of both heaps and semigroups. A subtruss $S$ of a truss $(T,[\cdot,\cdot,\cdot],\cdot)$ is subheap such that $S\cdot S\subset S.$

The following examples are given in \cite{Brz:par}.
\begin{exs}
         Let $(T,[\cdot,\cdot,\cdot],\cdot)$ be a truss and let $X$ be a set. The set $T^X$ of all functions $X\to T$ is a truss with the pointwise defined operations, i.e.\ for all $f,g,h\in T^X$,

\begin{equation}
[f,g,h] : X\to T, \qquad x\mapsto [f(x),g(x),h(x)],
\end{equation}
\begin{equation}
fg : X\to T, \qquad x\mapsto f(x)g(x).
\end{equation}
\end{exs}

%Trusses: Paragons, ideals and modules T. brezinski
\begin{ex}
Consider the Klein four group (written additively)
$V_4 = \{0,\, a,\, b,\, a+b\}$.
Then \(V_4\) becomes a truss when equipped with the heap operation \([\,\cdot,\cdot,\cdot\,]_+\) (induced by the group addition) and the multiplication defined by the table
\[
\begin{array}{c|cccc}
\cdot & 0   & a   & b   & a+b \\ \hline
0     & 0   & 0   & b   & a \\
a     & a   & a   & a+b & 0 \\
b     & b   & b   & 0   & a+b \\
a+b   & a+b & a+b & a+b & a\,b \\
\end{array}
\]
    
\end{ex}

\begin{ex}
 Let \(R\) be a ring and  \(e \in M_n(R)\) be an \(n \times n\) idempotent matrix (that is, \(e^2 = e\)). Consider the module \(R^n\), which forms an abelian group under addition. This additive structure induces the natural heap operation defined by
\[
[x,y,z] = x - y + z,\quad \text{for all } x,y,z \in R^n.
\]
Now, define a binary multiplication on \(R^n\) by
\[
(r_1,\ldots,r_n) \cdot (s_1,\ldots,s_n) = (r_1,\ldots,r_n) + (s_1,\ldots,s_n)e.
\]
One may verify that the multiplication is distributive over the heap operation, and hence 
$ \left(R^n, [\cdot,\cdot,\cdot], \, \cdot \, \right)$
forms a truss.   
\end{ex}

\begin{lem}\cite{Brz:par}
    The set of integers $\mathbb Z$ with the usual addition can be viewed as a heap with induced operation,
\begin{equation}\label{heap.+}
[l,m,n]_+ = l-m+n.
\end{equation}
There are two non-commutative truss structures on $(\mathbb Z,[\cdot,\cdot,\cdot]_+)$ with products defined for all $m,n\in \mathbb Z$,
\begin{equation}\label{mnm}
m\cdot n =m \quad \mbox{or} \quad m\cdot n = n. 
\end{equation}
\end{lem}
\begin{lem}\cite{Brz:par}\label{truss on Z}
 Up to isomorphism, the following commutative products equip \((\mathbb{Z}, [\cdot,\cdot,\cdot]_+)\) with different truss structures, for all \( m, n \in \mathbb{Z} \):
\begin{enumerate}
    \item For all \( a \in \mathbb{N} \),
    \begin{align}
        m \cdot n &= a m n, \quad a \neq 1, \label{eq:3.40a} \\
        m \cdot n &= a m n + m + n. \label{eq:3.40b}
    \end{align} 
    \item For all \( c \in \mathbb{N} \setminus \{0\} \),
    \begin{align}
        m \cdot n &= c, \label{eq:3.41a} \\
        m \cdot n &= m + n + c. \label{eq:3.41b}
    \end{align}
    \item For all \( a \in \mathbb{N} \setminus \{0,1,2\} \), \( b \in \{2,3,\dots,a-1\} \), and \( c \in \mathbb{N} \setminus \{0\} \) such that \( ac = b(b-1) \),
    \begin{equation}
        m \cdot n = a m n + b(m+n) + c. \label{eq:3.42}
    \end{equation}
\end{enumerate}  
% (2) on prend c pair et R(x)=-c/2 => R is a weighted RB op
\end{lem}
\begin{ex}\label{ClassifZ2}
In the following, we provide a classification of trusses on $\mathbb{Z}_2$ with the standard abelian heap ternary bracket  $[\cdot,\cdot,\cdot]_+.$

    Any truss product  on $\mathbb{Z}_2$,  with the standard abelian heap ternary bracket  $[\cdot,\cdot,\cdot]_+$, is isomorphic to  a truss corresponding to one of the binary operations given by the following Cayley tables providing the   multiplication of  the ith row elements by the jth column elements.    
    \begin{small}
    \begin{align*}
 (1):=&\begin{tabular}{|l|l|l|}
     \hline
     $\cdot$  & $0$ & $1$ \\
       \hline 
       $0$ & $0$   &$0$\\
      \hline 
      $1$ & $0$& $0$\\
      \hline
    \end{tabular},\;(2):=\begin{tabular}{|l|l|l|}
     \hline
     $\cdot$  & $0$ & $1$ \\
       \hline 
       $0$ & $0$   &$1$\\
      \hline 
      $1$ & $1$& $0$\\
      \hline
    \end{tabular},\; (3):=\begin{tabular}{|l|l|l|}
     \hline
     $\cdot$  & $0$ & $1$ \\
       \hline 
       $0$ & $0$   &$0$\\
      \hline 
      $1$ & $0$& $1$\\
      \hline
    \end{tabular},\; (4):=\begin{tabular}{|l|l|l|}
     \hline
     $\cdot$  & $0$ & $1$ \\
       \hline 
       $0$ & $0$   &$0$\\
      \hline 
      $1$ & $1$& $1$\\
      \hline
    \end{tabular}, \; 
(5):=\begin{tabular}{|l|l|l|}
     \hline
     $\cdot$  & $0$ & $1$ \\
       \hline 
       $0$ & $0$   &$1$\\
      \hline 
      $1$ & $0$& $1$\\
      \hline
    \end{tabular}.
    \end{align*}
    \end{small}
    \end{ex}

 Notice that the structures given by the tables:
 \begin{align*}(6):=\begin{tabular}{|l|l|l|}
     \hline
     $\cdot$  & $0$ & $1$ \\
       \hline 
       $0$ & $1$   &$0$\\
     \hline 
      $1$ & $0$& $1$\\
      \hline
   \end{tabular},\quad(7):=\begin{tabular}{|l|l|l|}
     \hline
     $\cdot$  & $0$ & $1$ \\
       \hline 
       $0$ & $0$   &$1$\\
      \hline 
      $1$ & $1$& $1$\\
      \hline
    \end{tabular}, \quad(8):=\begin{tabular}{|l|l|l|}
     \hline
     $\cdot$  & $0$ & $1$ \\
       \hline 
       $0$ & $1$   &$1$\\
      \hline 
      $1$ & $1$& $1$\\
      \hline
    \end{tabular}
\end{align*} are isomorphic to (2), (3) and (1), respectively.
\begin{defn}
    A left (resp. right) absorber element of a truss $T$ is an element ${\bf 0}\in T$ such 
that, for all $t \in T$, $t\cdot {\bf 0} = \bf 0$ (resp. ${\bf 0}\cdot t = \bf 0$). We say that $\bf 0$ is an absorber element  if it is a left and right absorber element.
\end{defn}

\begin{rmk}\label{TrussesRings}
Unless otherwise stated, by a \emph{ring} we mean an associative ring that is not necessarily unital. Given a ring $(R, +, \cdot)$, we can associate a truss $\mathcal{T}(R)$ where the heap structure is given by $\mathcal{H}(R, +)$, while the multiplication remains unchanged. The assignment $\mathcal{T}: R \to \mathcal{T}(R)$, acting as the identity on morphisms, defines a functor from the category of rings to the category of trusses.
Conversely, if a truss $T$ possesses an absorber element ${\bf 0} \in T$.
Then the structure $\mathcal{R}(T; {\bf 0}) := (\mathcal{G}(T; {\bf 0}), \cdot)$ forms a ring. Here, $\mathcal{G}(T; \bf 0)$ denotes the group-like structure derived from the truss $T$ with respect to the absorber $\bf 0$, and the multiplication operation is preserved. Moreover, this construction satisfies the identity:
\[
T = \mathcal{T}(\mathcal{R}(T; \bf 0)).
\]
This identity shows that the original truss $T$ can be recovered by applying the functor $\mathcal{T}$ to the ring $\mathcal{R}(T; \bf 0)$, establishing a natural correspondence between trusses with absorbers and rings.
\end{rmk}

\begin{lem}\label{LEMMA}
     In an Abelian heap $H$ we have, for all $x_i,y_i,z_i\in H$,
     \begin{equation}
         [[x_1,x_2,x_3],[y_1,y_2,y_3],[z_1,z_2,z_3]]=[[x_1,y_1
         ,z_1],[x_2,y_2,z_2],[x_3,y_3,z_3]].
     \end{equation}
\end{lem}
\begin{prop}\label{TrussTimes}
    Let $(T,[\cdot,\cdot,\cdot],\cdot)$  be a truss. Then the  set $T\times T$ has a  truss structure given by
    \begin{align}
        &\Courant{(a,x),(b,y),(c,y)}=([a,b,c],[x,y,z]),\\
       & (a,x)\odot (b,y)=(a\cdot b,[a\cdot y,{\bf 0},x\cdot b]),
    \end{align}
     for any $a,b,c,x,y,z\in T$. This truss is denoted by $T\ltimes T$.
    \end{prop}
    
    \begin{proof}
        Since $(T,[\cdot,\cdot,\cdot])$ is an abelian heap, then $(T\times T,[\cdot,\cdot,\cdot])$ is an abelian heap. Let $a,b,c,x,y,z\in T,$ we have
        \begin{align*}
         \big((a,x)\odot(b,y)\big)\odot(c,z)&=(a\cdot b,[a\cdot y,{\bf 0},x\cdot b])\cdot(c,z)\\
         &=((a\cdot b)\cdot c,[(a\cdot b)\cdot z,{\bf 0},[(a\cdot y)\cdot c,{\bf 0},(x\cdot b)\cdot c])\\
         &=(a\cdot( b\cdot c),[[a\cdot (b\cdot z),{\bf 0},a\cdot (y\cdot c)],{\bf 0},x\cdot (b\cdot c))\\
         &=(a,x)\cdot (b\cdot c,[b\cdot z,{\bf 0},y\cdot c])\\
         &=(a,x)\odot \big((b,y)\odot (c,z)\big).
        \end{align*}
        On the other hand, let $a,b,c,d,x,y,z,t\in T,$ we have
        \begin{align*}
            (a,x)\odot\Courant{(b,y),(c,z),(d,t)}&=(a,x)\odot([b,c,d],[y,z,t])\\
            &=([a\cdot b,a\cdot c,a\cdot d],[[a\cdot y,a\cdot z,a\cdot t],{\bf 0},[x\cdot b,x\cdot c,x\cdot d])\\
            &=([a\cdot b,a\cdot c,a\cdot d],[[a\cdot y,a\cdot z,a\cdot t],[{\bf 0},{\bf 0},{\bf 0}],[x\cdot b,x\cdot c,x\cdot d])\\
            &=([a\cdot b,a\cdot c,a\cdot d],[[a\cdot y,{\bf 0},x\cdot b ],[a\cdot z,{\bf 0},x\cdot c],[a\cdot t, {\bf 0},x\cdot d]]) \\&=[(a\cdot b,[a\cdot y,{\bf 0},x\cdot b ]),(a\cdot c,[a\cdot z,{\bf 0},x\cdot c]),(a\cdot d,[a\cdot t, {\bf 0},x\cdot d])]\\
            &=\Courant{(a,x)\odot(b,y),(a,x)\odot(c,z),(a,x)\odot(d,t)}.
        \end{align*}
        Therefore, $\odot$ is left distributive with respect to  the heap bracket. Similarly, we can check the  right distributivity.
    \end{proof}

More generally, let    $(T,[\cdot,\cdot,\cdot],\cdot)$  be a truss and $e$ be an element in the center  $($i.e. $e\cdot x=x\cdot e)$. Define the truss product   by
    \begin{align}\label{HeapLtimes}
       & (a,x)\boxdot_e (b,y)=(a\cdot b,[a\cdot y,e\cdot x\cdot y,x\cdot b]).
    \end{align}
    \begin{prop}
        Under the above notations, $(T\times T,[\cdot,\cdot,\cdot],\boxdot_e)$ is a truss with the heap structure given in \eqref{HeapLtimes}, which is denoted by $T\bowtie_e T.$
    \end{prop}

%%%%%%%%%%%%%%%%%%%%%%%%%%%%%%%%%%%%%%%%%%%%%%%%%%%%%%%
\section{Rota-Baxter operators on trusses}\label{Sec3}
%%%%%%%%%%%%%%%%%%%%%%%%%%%%%%%%%%%%%%%%%%%%%%%%%%%%%%%
In this section, we will assume that all trusses  are equipped with absorber elements, denoted by 
$\bf 0$. We introduce and define Rota-Baxter operators within the framework of trusses and provide a characterization of these operators using graph-theoretic methods. Furthermore, we explore the concept of dendriform trusses and investigate their connection to Rota-Baxter operators on trusses, analyzing the interplay and relationships between these structures.
\subsection{Definition, Examples and characterization}
\begin{defn} \label{RBDef}
    Let $(T,[\cdot,\cdot,\cdot],\cdot)$  be a truss.  A map $R:T\to T$ is called a Rota-Baxter  operator of weight $\bf 0$, where $\bf 0$ is an absorber element of $T$,  if it is a heap morphism such that $R({\bf 0})={\bf 0}$ and  satisfying 
    \begin{equation}
        \label{RBCon}
        R(x)\cdot R(y)=R[R(x)\cdot y,{\bf 0},x\cdot R(y)],\quad \forall x,y\in T.
    \end{equation}
\end{defn}
\begin{ex}
Consider the abelian heap $(\mathbb Z,[\cdot,\cdot,\cdot]_+)$ and let $a\in \mathbb Z$. According to Lemma 3.34 in \cite{Brz:par}, the binary operation defined by,
  $$m\cdot_a n = a,\quad \forall m, n \in \mathbb Z$$ 
 makes $(\mathbb Z,[\cdot,\cdot,\cdot]_+)$ into a   truss, in which $a$ is an absorber element. Define the map $R: \mathbb Z \to \mathbb Z$ by
 $$R(m)=[m,a,m]_+,\quad\forall m\in \mathbb Z.$$
 Then $R$ is a Rota-Baxter operator   on $(\mathbb Z,[\cdot,\cdot,\cdot]_+,\cdot_a)$.
\end{ex}

\begin{ex} 
 Observe that among the classes of trusses presented in Example \ref{ClassifZ2}, only classes $(1)$ and $(3)$ admit absorber elements. In the case of the nontrivial class $(3)$, the zero map is the only Rota-Baxter operator of weight $0$. 
\end{ex}

\begin{rmk}
 In \cite{TB1}, the authors introduced the definition of a derivation on a truss. It is well known  for algebras that an invertible derivation is a Rota-Baxter operator. However, the definition given in \cite{TB1} does not have this property. To ensure this one, we give a new definition of a derivation on a truss:   
 a derivation on a truss $(T,[\cdot,\cdot,\cdot],\cdot)$ is a  map $D:T\to T$ which is both a heap's morphism and satisfying 
    \begin{align}
        \label{derivation}
        & D({\bf 0})={\bf 0}, \\
      &  D(x\cdot y)=[D(x)\cdot y,{\bf 0},x\cdot D(y)],\quad \forall x,y\in T.
    \end{align}
    In this case, if $D$ is bijective, then $D^{-1}$ is a Rota-Baxter operator and vice-versa. 
\end{rmk}

\begin{prop}
  Let $(T,[\cdot,\cdot,\cdot],\cdot)$ be a truss, the map $R: T \to T$   is a Rota-Baxter operator of weight $\bf 0$ if and only if its graph 
  $${\sf Gr}(R):=\{(R(x),x),\  x\in T\}$$
  is a sub-truss of $T\ltimes T.$
\end{prop}
\begin{proof} 
It is easy to check that $R$ is a heap homomorphism if and only if ${\sf Gr}(R)$ is a sub-heap of $T \ltimes T$.
    Let $(R(x),x),(R(y),y)\in{\sf Gr}(R)$, we have
    $$(R(x),x)\odot(R(y),y)=(R(x)\cdot R(y),[R(x)\cdot y,{\bf 0},x\cdot R(y)]).$$
    Then, $R$   is a Rota-Baxter operator of weight $\bf 0$ if and only if its graph 
  ${\sf Gr}(R)$.
\end{proof}

\subsection{Derived and related structures}
\begin{thm}
  Let $(T,[\cdot,\cdot,\cdot],\cdot)$ be a truss and $R: T \to T$ be a Rota-Baxter operator of weight $\bf 0$. Define the product 
  \begin{align}
      & x \diamond y =[R(x)\cdot y,{\bf 0},x\cdot R(y)].
  \end{align}
  Then  $(T,[\cdot,\cdot,\cdot],\diamond)$ is also a truss with the same absorber element.
\end{thm}

\begin{proof}
Let $x,y,z,t \in T$. We begin by checking the associativity of $\diamond$. We have
\begin{align*}
  (x \diamond y) \dd z= &[R(x) \cdot y, {\bf 0}, x \cdot R(y)]\dd z \\
 =& [R(x)\cdot R(y) \cdot z,{\bf 0},[R(x)\cdot y \cdot R(z),{\bf 0},x \cdot R(y) \cdot R(z)]]
\end{align*}
and 
\begin{align*}
    x \dd (y \dd z)=& x\dd [R(y)\cdot z,{\bf 0},y \cdot R(z)] \\
    =& [[R(x)\cdot R(y) \cdot z,{\bf 0},R(x)\cdot y \cdot R(z)],{\bf 0},x \cdot R(y) \cdot R(z)].
\end{align*}
Using the associativity of the heap bracket $[\cdot,\cdot,\cdot]$, we conclude that $(x\dd y)\dd z=x \dd (y \dd z)$.
Now, it remains to prove the distributivity  laws. Using Lemma \ref{LEMMA} and since $[{\bf 0},{\bf 0},{\bf 0}]={\bf 0}$, we have 
 \begin{align*}
     x \dd [y,z,t]=& [[R(x)\cdot y,R(x)\cdot z,R(x)\cdot t],{\bf 0},[x\cdot R(y), x \cdot R(z), x \cdot R(t)] ]\\
     =& [[R(x)\cdot y,{\bf 0},x \cdot R(y)],[R(x)\cdot z,{\bf 0},x \cdot R(z)],[R(x)\cdot t,{\bf 0},x \cdot R(t)]] \\
     =& [x \dd y, x \dd z,x \dd t].
 \end{align*}
 The right distributivity can be similarly  checked  and the proof is finished. 
\end{proof}

Now, we introduce the notion of dendriform truss which is a splitting version of trusses, as we will see later. In addition, it leads to a generalized concept of dendriform rings in the same way as trusses can be considered as a generalization of rings. Moreover, the correspondence between trusses and rings is invariant by dendrification. In the literature, the notion of dendriform ring was  not defined although that of dendriform algebra was introduced by Loday. So we first give the definition of a dendriform truss. 

\begin{defn}
A dendriform truss is a tuple $(T,[\cdot,\cdot,\cdot],\prec,\succ,{\bf 0})$, where $(T,[\cdot,\cdot,\cdot])$ is an abelian heap,  ${\bf 0}\in T$ is a distinguished element and $\prec,\succ: T \otimes T \to T$ are two binary products satisfying:
\begin{align}
&x \succ {\bf 0}={\bf 0} \succ x={\bf 0} \prec x=x \prec {\bf 0}={\bf 0}, \label{dend1} \\
& [x,y,z]\succ t=[x\succ t,y\succ t,z \succ t],\quad 
x \prec [y,z,t]=[x\prec y,x \prec z,x \prec t],
\label{dend2}\\
 &   [x \succ y,{\bf 0},x \prec y] \succ z= x \succ (y \succ z), \label{dend3} \\
  & (x \succ y) \prec z=  x \succ (y \prec z), \label{dend4} \\
  & (x \prec y) \prec z=x \prec  [y \succ z , {\bf 0}, y \prec z ], \label{dend5}
\end{align}
for all $x,y,z \in T$. Similarly to trusses, the element $\bf 0$ is called the absorber element of the dendriform truss $T$. 
\end{defn}

\begin{thm}\label{DendTrusToTrus}
 Let $(T,[\cdot,\cdot,\cdot],\prec,\succ,{\bf 0})$  be a dendriform truss. Then $(T,[\cdot,\cdot,\cdot],\circ)$ is a truss, where
 \begin{align}
  &  x \dd y= [x \succ y,{\bf 0},x \prec y].  
 \end{align}
The tuple $(T,[\cdot,\cdot,\cdot],\dd)$ is called the subadjacent  truss of  $(T,[\cdot,\cdot,\cdot],\prec,\succ,{\bf 0})$.
\end{thm}

\begin{proof}
 Let $x,y,z,t \in T$. We start by proving the associativity of  the product $\dd$. 
 \begin{align*}
  (x \dd y)\dd z=& [x\succ y,{\bf 0},x \prec y]\dd z \\
  =& [[x\succ y,{\bf 0},x \prec y]\succ z,{\bf 0},[x\succ y,{\bf 0},x \prec y]\prec z] \\
  =& [ x \succ (y \succ z),  {\bf 0},  [(x \succ y)\prec z,{\bf 0},(x\prec y)\prec z]    ] \\
    =& [ [x \succ (y \succ z),  {\bf 0},  x \succ (y\prec z)],{\bf 0},(x\prec y)\prec z    ] \\
    =& [x \succ [y \succ z,{\bf 0},y \prec z],{\bf 0},x\prec [y \succ z,{\bf 0},y \prec z]   ] \\
    =& x \dd (y\dd z).
 \end{align*}
 On the other hand, we have
 \begin{align*}
   x \dd [y,z,t]=& [x \succ [y,z,t],{\bf 0},x \prec [y,z,t]] \\
   =& [[x \succ y,x \succ z,x \succ t], {\bf 0}, [x \prec y,x \prec z,x\prec t] ] \\
   =&[ [x\succ y,{\bf 0},x \prec y],[x\succ z,{\bf 0},x \prec z], [x\succ t,{\bf 0},x \prec t] ]\quad  (\text{Lemma}\  \ref{LEMMA}) \\
   =& [x \dd y, x \dd z, x \dd t].
 \end{align*}
 The right distributivity can be proved in the same way. 
\end{proof}    

Conversely, if we have a truss $T$, we can get a dendriform truss  on the same heap structure of $T$ by means of a Rota-Baxter operator. 
\begin{thm}\label{TrusToDendTrus}
   Let $(T,[\cdot,\cdot,\cdot],\cdot)$ be a truss and $R: T \to T$ be a Rota-Baxter operator of weight $\bf 0$. 
  Then the binary products
  \begin{align} \label{dend truss}
      &  x \succ y= R(x) \cdot y, \quad x \prec y= x \cdot R(y)
  \end{align}
  carry a dendriform truss structure on $T$ with absorber element ${\bf 0}$. In addition, $R$ is a truss morphism.
  
\end{thm}

\begin{proof}
Let $x,y,z,t \in T$. First note that since $R({\bf 0})={\bf 0}$, then Equation \eqref{dend1} holds. In addition, we have 
\begin{align*}
    [x,y,z] \succ t=& [R(x),R(y),R(z)]\cdot t \\
    =& [R(x) \cdot t,R(y) \cdot t, R(z) \cdot t] \\
    =& [x \succ t, y \succ t, z \succ t].
\end{align*}
Then the first identity   \eqref{dend2} holds and the second can be done similarly. 
Furthermore, 
\begin{align*}
  [x \succ y, {\bf 0}, x \prec y] \succ z =& [R(x)\cdot y,{\bf 0}, x \cdot R(y)] \succ z \\
=& R[R(x)\cdot y,{\bf 0}, x \cdot R(y)] \cdot z  \\
=& (R(x) \cdot R(y) ) \cdot z  \\
=& R(x) \cdot (R(y) \cdot z) =x \succ (y \succ z)
\end{align*}
and then Equation \eqref{dend3} holds. To check Equation \eqref{dend4}, we compute as follows
\begin{align*}
   (x \succ y) \prec z=& (R(x) \cdot y) \prec z \\
   =& (R(x) \cdot y) \cdot R(z) \\
   =& R(x) \cdot ( y \cdot R(z))=x \succ (y \prec z). 
\end{align*}
The identity \eqref{dend5} can be proved in a similar way as identity \eqref{dend3}. 
\end{proof}

\begin{defn}
Let $(A,+,0)$ be an abelian group   and $\succ,\prec: A \otimes A \to A$ be two binary products. The triple $(A,+,\succ,\prec)$ is called a dendriform ring provided that:
\begin{align*}
&x\succ 0=0\prec x=0,\\
& (x+y)\succ z=x\succ z+y\succ z,\quad 
x \prec (y+z)=x\prec y+x \prec z,
\\
  & (x \succ y+x \prec y) \succ z=x \succ (y \succ z),\\
  & (x \succ y) \prec z=x \succ (y \prec z),\\
  &(x \prec y) \prec z= x \prec (y \succ z+y \prec z),
\end{align*}
for all $x,y,z \in A$. 
\end{defn}

\begin{rmk}
 Note that if $(A,+,\succ,\prec)$ is a dendriform ring, then $(A,+,\times)$ is a ring called the subadjacent ring, where $$ x \times y= x \succ y + x \prec y,\quad \forall x,y \in A.$$   
\end{rmk}

\begin{thm}
 There is a one-to-one correspondence between dendriform rings and dendriform trusses.    
\end{thm}

\begin{proof}
Let $(T,[\cdot,\cdot,\cdot],\prec,\succ,{\bf 0})$ be a dendriform truss. Define 
$$ x + y =[x,{\bf 0},y], \forall x,y \in T.$$
Then $(T,+)$ is an abelian group with a neutral element ${\bf 0}$. 
We can easily  verify that $(T,+,\succ,\prec)$ is a dendriform ring. 
On the other hand, let $(A,+,\succ,\prec)$ be a dendriform ring and denote by $0$ the neutral element of the addition. Define the ternary product
$$ [x,y,z]=x-y+z,\ \forall x,y,z \in A.$$
Then $(A,[\cdot,\cdot,\cdot])$ is an abelian heap and we can check that $(A,[\cdot,\cdot,\cdot],\succ,\prec,0)$ is a dendriform truss. 
\end{proof}

The notion of Rota-Baxter operator on rings haven't been introduced yet  despite its importance. So we give its definition:
Let $(A,+,\times)$ be a ring and $R: A \to A $ be a group homomorphism. We say that $R$ is a Rota-Baxter operator on $A$ if it satisfies the following identity 
\begin{equation}
    \label{RBRing}R(x)\times R(y)=R(R(x)\times y+x\times R(y)),
\end{equation}
for all $x,y \in A$. The following observation is straightforward.
\begin{prop}
 Let   $(A,+,\times)$ be a ring and $R: A \to A $ be a Rota-Baxter operator. Then the following products
 \begin{equation}
     \label{ProductDendRing}x\succ y=R(x)\times y\quad \text{and}\quad x\prec y=x\times R(y),\quad \forall x,y\in A,
 \end{equation}
 provide a dendriform ring structure on $A$. 
\end{prop}
\begin{prop}
  Let   $(A,+,\times)$ be a ring and $R: A \to A $ be a Rota-Baxter operator, then $R$  is  a Rota-Baxter operator on the associated truss $T(A)$. Conversely, If $R$ is a Rota-Baxter operator on a truss $T$, then it is still a Rota-Baxter operator on its associated ring. 
\end{prop}
\begin{proof}
    Let   $(A,+,\times)$ be a ring and $R$ be a Rota-Baxter operator. Since $R$ is a group homomorphism, then it is also a  morphism for the associated heap. On the other hand,   
    let $x,y\in A$. Then 
\begin{align*}
    R([R(x)\times y,0,x\times R(y)])&=R(R(x)\times y-0+x\times R(y))\\
    &=R(x)\times  R(y).
\end{align*}
  Conversely,   let $R$ be a Rota-Baxter operator on a truss $(T,[\cdot,\cdot,\cdot],\cdot)$ and denote by $(A,+,\cdot)$ 
the associated ring. 
It is clear that $R$ is a group morphism for $(A,+)$. In addition, for all $x,y \in A$, we have   
\begin{align*}
  R(R(x)\cdot y+x\cdot R(y)) 
  &= R([R(x)\cdot y,{\bf0},x\cdot R(y)])\\
  &= R(x)\cdot R(y).
\end{align*}
\end{proof} 

\begin{coro}
 Let $(A,+,\succ,\prec)$ be a dendriform ring and define $$[x,y,z]=x-y+z,\ \forall x,y,z \in A.$$   
Then $(T,[\cdot,\cdot,\cdot],\succ,\prec)$ is a dendriform truss. 
\end{coro}

\begin{coro}
 Let $(T,[\cdot,\cdot,\cdot],\succ,\prec)$ be a dendriform truss. Suppose that there is an absorber  element $a \in T$ and define the operation
 $$x+y=[x,a,y],\ \forall x,y, \in A.$$
 Then $(A,+,\succ,\prec)$ is a dendriform ring.
\end{coro}

\subsection{Weighted Rota-Baxter operators}
In the sequel, we introduce the notion of weighted  Rota-Baxter operator  on a truss. Such operator splits the truss structure into three pieces which constitute what we will call a tridendriform truss.

\begin{defn}
Let $(T,[\cdot,\cdot,\cdot],\cdot)$ be a truss and $a\in T$ be an element in the center with respect to the binary operation. A heap homomorphism $R: T \to T$ is called a weighted  Rota-Baxter of weight $a$ on $T$  if it satisfies
\begin{align}
 & R(x)\cdot R(y)=R([R(x)\cdot y  ,a\cdot x \cdot y,x \cdot R(y)]
\end{align}
for all $x,y \in T$. 
\end{defn}

This definition is in the same spirit as the original definition of Rota-Baxter operators given by F. V. Atkinson in \cite{Atkinson}, see also more recent observations and developments in \cite{Goncharo}.
If the truss admits an absorber element $\mathbf{0}$, we recover for $a=\mathbf{0}$ the  Rota-Baxter operators of weight $\mathbf{0}$ studied in Section \ref{Sec3}.
If the truss is unital with $1$ as a unit element, setting $a=1$, one gets Rota-Baxter operator of weight $1$.

\begin{defn}
Let  $(T,[\cdot,\cdot,\cdot],\cdot)$ be a truss (not necessarily unital). We say that a heap homomorphism $R: T \to T$ is  a  Rota-Baxter operator of weight $1$ on $T$  if it satisfies
\begin{align}
 & R(x)\cdot R(y)=R([R(x)\cdot y  , x \cdot y,x \cdot R(y)].
\end{align}
\end{defn}
\begin{ex}\label{RB on Z2}
 Let $\mathbb{Z}_2$ be the truss with the heap operation 
$[\cdot, \cdot, \cdot]_+$ and
 multiplication  given in Example \ref{ClassifZ2} by Table $(3)$. Define the map $R:\mathbb{Z}_2\to \mathbb{Z}_2$ by $R(1)=0$ and $R(0)=1.$ Then $R$ is  a  Rota-Baxter operator of weight $1$ on $\mathbb{Z}_2.$

\end{ex}

\begin{prop}
  Let $(T,[\cdot,\cdot,\cdot],\cdot)$ be a truss, the map $R: T \to T$   is a weighted Rota-Baxter operator of weight $a$ if and only if its graph 
  $${\sf Gr}(R):=\{(R(x),x),\  x\in T\}$$
  is a sub-truss of $T\bowtie_a T.$
\end{prop}
\begin{proof} 
It is easy to check that $R$ is a heap homomorphism if and only if ${\sf Gr}(R)$ is a sub-heap of $T \bowtie T$.
    Let $(R(x),x),(R(y),y)\in{\sf Gr}(R)$, we have
    $$(R(x),x)\boxdot_a(R(y),y)=(R(x)\cdot R(y),[R(x)\cdot y,a\cdot x\cdot y,x\cdot R(y)]).$$
    Then, $R$   is a  weighted Rota-Baxter operator if and only if its graph 
  ${\sf Gr}(R)$  is a sub-truss of $T\bowtie_a T$.
\end{proof}
\begin{defn}
 Let $(T,[\cdot,\cdot,\cdot])$ be an abelian heap and $\succ,\vee,\prec: T \otimes T \to T$ be three binary operations which are distributive with respect to the heap structure. We say that $(T,[\cdot,\cdot,\cdot],\succ,\vee,\prec)$ is a tridendriform truss if it satisfies the following set of identities: 
 \begin{align}
&(x\prec y)\prec z =x\prec [y\succ z, y \vee z,y \prec z], \label{tri1}\\
&(x\succ y)\prec z =x\succ(y\prec z), \label{tri2}\\
&[x \succ y, x \vee y,x \prec y]\succ z =x \succ(y\succ z), \label{tri3} \\
&(x\succ y)\vee z =x\succ(y\vee z), \label{tri4} \\
&(x\prec y)\vee z =x\vee (y \succ z), \label{tri5} \\
&(x\vee y)\prec z =x\vee (y\prec z), \label{tri6} \\
&(x\vee y)\vee z =x\vee (y \vee  z ) \label{tri7},
 \end{align}
 for all $x,y,z \in T$. 
\end{defn}

\begin{prop}
Let $(T,[\cdot,\cdot,\cdot],\succ,\vee,\prec)$ be a tridendriform truss. Then $(T,[\cdot,\cdot,\cdot],\dd)$ is a truss, where
\begin{align}
    & x \dd y= [x \succ y, x \vee y,x \prec y],\quad \forall x,y \in T.
\end{align}
\end{prop}

\begin{proof}
Let $x,y,z \in T$. We have
{\small\begin{align*}
& x \dd (y \dd z)= x \dd [y \succ z,y \vee z,y \prec z] \\
=& [x \succ [y \succ z,y \vee z,y \prec z], x \vee [y \succ z,y \vee z,y \prec z], x \prec [y \succ z,y \vee z,y \prec z]   ] \\
=&[[x \succ (y \succ z),x \succ (y \vee z),x \succ (y \prec z)], [x \vee (y \succ z),x \vee (y \vee z),x \vee(y \prec z)], (x \prec y)\prec z ] \\
=&[ [[x\succ y,x\vee y,x \prec y]\prec z, (x\succ y)\vee z, (x\succ y)\prec z] , [(x\prec y)\vee z,(x\vee y)\vee z,(x\vee y)\prec z ],   (x \prec y)\prec z    ]  \\
=&[[x\succ y,x\vee y,x \prec y]\prec z, (x\succ y)\vee z,
[(x\succ y) \prec z, [(x\prec y)\vee z,(x\vee y)\vee z,(x\vee y)\prec z ],   (x \prec y)\prec z]] \\
=& [[x\succ y,x\vee y,x \prec y]\succ  z, (x\succ y)\vee z,
[[(x\succ y) \prec z, (x\prec y)\vee z,(x\vee y)\vee z],(x\vee y)\prec z ],   (x \prec y)\prec z] \\ 
=&  [[x\succ y,x\vee y,x \prec y]\succ  z, (x\succ y)\vee z,
[[  (x\vee y)\vee z, (x\prec y)\vee z, (x\succ y) \prec z ],  (x\vee y)\prec z ],   (x \prec y)\prec z] \\ 
=&  [[x\succ y,x\vee y,x \prec y]\succ  z, (x\succ y)\vee z,
[  (x\vee y)\vee z, (x\prec y)\vee z, [(x\succ y) \prec z ,  (x\vee y)\prec z ,   (x \prec y)\prec z]  ] \\
=& [[x\succ y,x\vee y,x \prec y]\succ z , [x\succ y,x\vee y,x \prec y]\vee z, [x\succ y,x\vee y,x \prec y]\prec z  ] \\
=& (x \dd y) \dd z. 
\end{align*}}
The distributivity of $\dd$ with respect to the heap structure is straightforward. 
\end{proof}
    
\begin{thm}\label{truss to tritruss}
      Let $(T,[\cdot,\cdot,\cdot],\cdot)$ be a truss and $R: T \to T$ be a weighted Rota-Baxter operator. Define the products 
  \begin{align}
      & x \succ y =R(x)\cdot y,  \quad  x \prec y=x\cdot R(y),\quad x\vee y=a\cdot x\cdot y,\ \forall x,y \in T.
  \end{align}
  Then  $(T,[\cdot,\cdot,\cdot],\succ,\vee,\prec)$ is  a tridendriform truss. In addition, if $(T,[\cdot,\cdot,\cdot],\dd)$ is the subadjacent truss, $R$ is a truss homomorphism from $(T,[\cdot,\cdot,\cdot],\dd)$ to $(T,[\cdot,\cdot,\cdot],\cdot)$. 
\end{thm}
\begin{ex} 
Consider the truss structure on $\mathbb Z_2$ given by table $(3)$ in Example \ref{ClassifZ2} and the Rota-Baxter operator given in Example \ref{RB on Z2}. Then, according to Theorem \ref{truss to tritruss}, the binary operations given by 
    $$ \begin{tabular}{|l|l|l|}
     \hline
     $\vee$  & $0$ & $1$ \\
       \hline 
       $0$ & $0$   &$0$\\
      \hline 
      $1$ & $0$& $1$\\
      \hline
    \end{tabular},\quad \quad \begin{tabular}{|c|l|l|}
     \hline
     $\succ$  & $0$ & $1$ \\
       \hline 
       $0$ & $0$   &$1$\\
      \hline 
      $1$ & $0$& $0$\\
      \hline
    \end{tabular} \quad 
    \text{ and }\quad 
    \begin{tabular}{|c|l|l|}
     \hline
     $\prec$  & $0$ & $1$ \\
       \hline 
       $0$ & $0$   &$0$\\
      \hline 
      $1$ & $1$& $0$\\
      \hline
    \end{tabular} 
    $$
endow  the abelian heap $\mathbb Z_2$ with a tridendriform truss structure. 
\end{ex}

\begin{proof}
 Let $x,y,z \in T$.  We will just prove identities \eqref{tri1}    and \eqref{tri4} and the others can be shown by similar computations. We have 
 \begin{align*}
   x \prec [y \succ z,y \vee z,y \prec z]=& 
   x \prec [R(y)\cdot z,a\cdot y \cdot z,y\cdot R(z)] \\
   =& x \cdot R ([R(y)\cdot z,a \cdot y\cdot z,y\cdot R(z)] ) \\
   =& x \cdot (R(y) \cdot R(z)) \\
   =& (x \cdot R(y)) \cdot R(z) \\
   =& (x \prec y) \prec z 
 \end{align*}
 and
\begin{align*}
&(x\succ y)\vee z =a\cdot(R(x)\cdot y)\cdot z= R(x) \cdot(a\cdot y\cdot z) = x\succ(y\vee z).  \end{align*}
  
Then Equations \eqref{tri1}        and \eqref{tri4} holds. 
 \end{proof}
Recall that, a tridendriform ring is a quardruple $ (A,\prec ,\succ ,\vee) $  consisting of three  multiplications  maps $\succ ,\prec ,\vee : A \times A \to A$, which satisfy the following axioms:
	\begin{equation}
	\begin{aligned}&(a\prec b)\prec c =a\prec (b\ast c),\\&(a\succ b)\prec c=a\succ (b\prec c),\\&(a\ast b)\succ c =a\succ (b\succ c),\\&(a\succ b)\vee c=a\succ (b\vee c),\\&(a\prec b)\vee c=a \vee (b\succ c),\\&(a \vee b)\prec c =a\vee(b\prec  c),\\&(a\vee b)\vee c=a\vee(b\vee c),\end{aligned}
	\end{equation} where $\ast=\succ-\vee+\prec$, for $a,b,c \in A$.
\begin{coro}
 Let $(A,+,\succ,\vee,\prec)$ be a tridendriform ring and define $$[x,y,z]=x-y+z,\ \forall x,y,z \in A.$$   
Then $(T,[\cdot,\cdot,\cdot],\succ,\vee,\prec)$ is a tridendriform truss. 
\end{coro}

\begin{coro}Let 
 $(T,[\cdot,\cdot,\cdot],\succ,\vee,\prec)$ be a tridendriform truss. Suppose that there is an absorber  element $a \in T$ and define the operation
 $$x+y=[x,a,y],\ \forall x,y, \in A.$$
 Then $(A,+,\succ,\vee,\prec)$ is a tridendriform ring.
\end{coro}

%%%%%%%%%%%%%%%%%%%%%%%%%%%%%%%%%%%%%%%%%%%%
\section{Reynolds  and Nijenhuis operators on trusses and NS-trusses}\label{Sec4}
%%%%%%%%%%%%%%%%%%%%%%%%%%%%%%%%%%%%%%%%%%%
In this section, we introduce the notion of Reynolds operators on trusses. These operators allow us to split a truss structure into three parts and get another generalization of trusses called NS-trusses. All properties proved for Rota-Baxter operators in the previous section hold for Reynolds operators.

%=====================================
\subsection{NS-rings and NS-trusses}
%====================================

In the  following,
we introduce the notion of NS-rings and NS-trusses. 

\begin{defn}
  An NS-ring is an abelian group $(A,+)$ together  with three binary operations $\succ, \prec, \ns: A \otimes A \to A$ that are 
  distributive with respect to addition and  satisfying the following set of identities:
   \begin{align}
   & (x \prec y)\prec z=x\prec (y \bullet z),\\
   & (x \succ y) \prec z=x \succ (y \prec z),\\
   & (x \bullet y) \succ z= x \succ (y \succ z),\\
   & ( x \ns y) \prec z + (x \bullet y) \ns z= x \succ (y \ns z)+x \ns (y \bullet z) ,
   \end{align}
   for all $x,y,z \in A$ and  where $x \bullet y= x \succ y+x \prec z + x \ns y$.  We will denote such a dendriform ring by $(A,\prec,\ns,\prec)$. 
\end{defn}
\begin{rmk}\label{NSRingToRing}
Exactly as for associative algebras, we can check that $(A,+,\bullet)$ is  also a ring. Moreover, if the binary operation $\ns$ is trivial, then the tuple $(A,+,\succ,\prec)$  is a dendriform ring provided that the absorption axiom \eqref{dend1}  holds.    
\end{rmk}

Now, we generalize the notion of NS-rings to the truss-setting. 
\begin{defn}
 An NS-truss is an abelian heap $(T,[\cdot,\cdot,\cdot])$ together with three binary operations $\succ,\prec$ and $\ns: T \otimes T \to T$ satisfying the following axioms:
    \begin{align}
   & (x \prec y)\prec z=x\prec [y \succ z, y \ns z,y \prec z], \label{NS1}\\
  \label{NS2} & (x \succ y) \prec z=x \succ (y \prec z),\\
 \label{NS3}  & [x \succ y, x\ns y,x \prec y] \succ z= x \succ (y \succ z),\\
 \label{NS4}  & ( x \ns y) \prec z  = x \succ (y \ns z) ,\\
  \label{NS5} &   [x \succ y, x\ns y,x \prec y] \ns z= x \ns [y \succ z, y \ns z,y \prec z],
   \end{align}
  for all $x,y,z \in T$. This dendriform truss will be denoted by $(T,[\cdot,\cdot,\cdot],\succ,\ns,\prec)$. 
\end{defn}

\begin{thm}\label{NSTrusToNSRing}
   Let $(T,[\cdot,\cdot,\cdot],\succ,\ns,\prec)$ be an NS-truss. Then the binary product
   \begin{align}
     \label{dend truss to truss}
     & x \dd y=[x \succ y,x \ns y,x \prec y],\quad x,y \in T,
   \end{align}
   provides a truss structure on the abelian heap $(T,[\cdot,\cdot,\cdot])$. 
\end{thm}

\begin{proof}
  Let $x,y,z,t \in T$. We will first check the associativity:
  \begin{align*}
 (x\dd y) \dd z=& [(x \dd y)\succ z,(x \dd y) \ns z,(x \dd y) \prec z] \\
 =& [x \succ (y\succ z),x \ns (y \dd z),[(x \succ y)\prec z,(x \ns y)\prec z,(x \prec y)\prec z] ]
  \end{align*}
and 
  \begin{align*}
 x\dd (y \dd z)=& [x \succ (y\dd z),x \ns( y \dd z),x \prec (y \dd z)] \\
 =& [ [x \succ(y \succ z),x \succ(y \ns z),x \succ(y \prec z)],(x \dd y) \ns z,(x \prec y) \prec z]\\
  =& [ x \succ(y \succ z),[x \succ(y \ns z),x \succ(y \prec z),(x \dd y) \ns z],(x \prec y) \prec z]\\
  =& [ x \succ(y \succ z),[(x \dd y) \ns z,(x \succ y) \prec z,x \succ(y \ns z)],(x \prec y) \prec z]\\
  =& [ x \succ(y \succ z),(x \dd y) \ns z,[(x \succ y) \prec z,x \succ(y \ns z),(x \prec y) \prec z]] .
  \end{align*}
  Then $(x \dd y)\dd z=x \dd (y \dd z)$.  The distributivity of $\dd$ with respect to $[\cdot,\cdot,\cdot]$ is immediate. 
 \end{proof}
%%%%%%%%%%%%%%%%%%%%%%%%%%%%%%%%%%%%%%%%%%%%
%\section{Reynolds  operators and Nijenhuis operators on trusses}\label{Sec4}
%%%%%%%%%%%%%%%%%%%%%%%%%%%%%%%%%%%%%%%%%%%
%In this section, we introduce the notion of Reynolds operators on trusses. These operators allow us to split a truss structure into three parts and get another generalization of trusses called NS-trusses. All properties proved for Rota-Baxter operators in the previous section are conserved for Reynolds operators. 
%======================================
\subsection{Reynolds operators on trusses}
%========================================
The goal of this subsection is to introduce the notion of Reynolds operator on trusses. Such operators allow us to split a truss multiplication into three compatible parts which give an NS-truss structure.
\begin{defn} \label{ReynoldDef}
    Let $(T,[\cdot,\cdot,\cdot],\cdot)$ be a truss.  A map $P:T\to T$ is called a Reynolds operator if it is a heap morphism of  $(T,[\cdot,\cdot,\cdot])$ and satisfying
    \begin{equation}
        \label{ReynoldCon}
        P(x)\cdot P(y)=P[P(x)\cdot y,P(x)\cdot P(y),x\cdot P(y)],\quad \forall x,y\in T.
    \end{equation}
\end{defn}

\begin{ex}
In \cite{TB1}, the author gave a definition of a derivation on a truss as follows.  We call it a modified derivation in this paper. A modified derivation on a truss $(T,[\cdot,\cdot,\cdot],\cdot)$ is a heap morphism $D: T \to T$ which satisfies the following Leibniz-like rule:
\begin{align}
  D(x\cdot y)=[x\cdot D(y), x\cdot y, D(x) \cdot y], \forall x,y \in T.   
\end{align}
Then we have the following equivalence:
$D$ is a modified derivation if and only if $D^{-1}$ is a Reynolds operator. 
\end{ex}

\begin{thm}
  Let $(T,[\cdot,\cdot,\cdot],\cdot)$ be a truss and $P: T \to T$ be a Reynolds operator. Define the product 
  \begin{align}
      & x \circ y =[P(x)\cdot y,P(x)\cdot P(y),x\cdot P(y)].
  \end{align}
  Then $(T,[\cdot,\cdot,\cdot],\circ)$ is also a truss.
\end{thm}

\begin{proof}
 Let $x,y,z,t \in T$. Then 
   { \small
 \begin{align*}
 x \circ [y,z,t]=&[ [P(x)\cdot y,P(x)\cdot z,P(x)\cdot t],[P(x)\cdot P(y),P(x)\cdot P(z),P(x)\cdot P(t)],[x\cdot P(y),x\cdot P(z),x \cdot P(t)] ]    \\
 =& [ [P(x)\cdot y,P(x)\cdot P(y),x\cdot P(y)],[P(x)\cdot z,P(x)\cdot P(z),x\cdot P(z)],[P(x)\cdot t,P(x)\cdot P(t),x \cdot P(t)] ] \\
 =&[x \circ y,x\circ z,x \circ t]. 
 \end{align*}
 }
 Similarly, we show that $[y,z,t]\circ x=[y\circ x,z \circ x,t \circ x]$. It remains to prove the associativity of $\circ$. We have
 {\small
 \begin{align*}
 (x \circ y)\circ z=&[P(x)\cdot y,P(x)\cdot P(y),x\cdot P(y)] \circ z \\
 =& [(P(x)\cdot P(y)) \cdot z,(P(x)\cdot P(y))\cdot P(z),
 [(P(x)\cdot y)\cdot P(z),(P(x)\cdot P(y))\cdot P(z),(x\cdot P(y))\cdot P(z)] ] \\
 =& [P(x)\cdot (P(y) \cdot z),P(x)\cdot ( P(y)\cdot P(z)),
 [P(x)\cdot (y\cdot P(z)),P(x)\cdot (P(y)\cdot P(z)),x\cdot (P(y)\cdot P(z))] ] \\
  =&[ [P(x)\cdot (P(y) \cdot z),P(x)\cdot ( P(y)\cdot P(z)),
 P(x)\cdot (y\cdot P(z))] ,P(x)\cdot (P(y)\cdot P(z)),x\cdot (P(y)\cdot P(z)) ]  \\
 =&[ P(x)\cdot[ P(y) \cdot z, P(y)\cdot P(z),
 y\cdot P(z)] ,P(x)\cdot (P(y)\cdot P(z)),x\cdot (P(y)\cdot P(z)) ] \\
 =& [P(x) \cdot (y \circ z), P(x)\cdot P(y \circ z), x \cdot P(y \circ z) ] \\
 =& x \circ (y \circ z).
 \end{align*}
 }
\end{proof}

\begin{thm}
  Let $(T,[\cdot,\cdot,\cdot],\cdot)$ be a truss and $P: T \to T$ be a Reynolds operator. Define the products 
  \begin{align}
      & x \succ y =P(x)\cdot y,\quad x \ns y=P(x)\cdot P(y),\quad  x \prec y=x\cdot P(y),\ \forall x,y \in T.
  \end{align}
  Then  $(T,[\cdot,\cdot,\cdot],\succ,\ns,\prec)$ is  an NS-truss. In addition if  $(T,[\cdot,\cdot,\cdot],\dd)$ is the subadjacent truss, the $P$ is truss homomorphism from  
  $(T,[\cdot,\cdot,\cdot],\dd)$ to $(T,[\cdot,\cdot,\cdot],\cdot)$. 
\end{thm}

\begin{proof}
For all $x,y,z \in T$, we have:
\begin{align*}
 x \prec[y \succ z,y \ns z, y\prec z]=& x \prec [P(y)\cdot z, P(y)\cdot P(z), y \cdot P(z)] \\
 =& x \cdot P([P(y)\cdot z, P(y)\cdot P(z), y \cdot P(z)]) \\
 =& x \cdot (P(y) \cdot P(z)) \\
 =& (x \cdot P(y)) \cdot P(z) \\
 =& (x \prec y) \prec z. 
\end{align*}
Hence we have proved Equation \eqref{NS1}. To show Equation \eqref{NS2}, we compute as follows
\begin{align*}
 (x \succ y) \prec z= & (P(x) \cdot y) \prec z \\
 =& (P(x) \cdot y) \cdot P(z) \\
 =& P(x) \cdot (y \cdot P(z) )= x \succ (y \prec z). 
\end{align*}
In addition, we have 
\begin{align*}
 [x \succ y,x \ns y,x \prec y]\ns z=& [P(x)\cdot y, P(x)\cdot P(y),x \cdot P(y)] \ns z \\
 =& P([P(x)\cdot y, P(x)\cdot P(y),x \cdot P(y)]) \cdot P(z)\\
 =& (P(x) \cdot P(y)) \cdot P(z) \\
 =& P(x) \cdot (P(y) \cdot P(z)) \\
 =& x \ns [y \succ z,y \ns z,y \prec z],
\end{align*}
which implies Equation  \eqref{NS5}. Identities \eqref{NS3} and \eqref{NS4} can be similarly proved. 
 \end{proof}
In  the next paragraph, we recall the notion of Nijenhuis operator on a truss, introduced in \cite{TB2},  and provide their connection with NS-trusses.

%===============================================
\subsection{Nijenhuis operators on trusses}
%=================================================
\begin{defn}
     Let $(T,[\cdot,\cdot,\cdot],\cdot)$ be a truss.  A heap homomorphism $N: T \to T$ is called a Nijenhuis operator if it satisfies the following integrability condition:
     \begin{align}
   & N(x)\cdot N(y)=N([N(x)\cdot y, N(x\cdot y), x\cdot N(y)])  ,
     \end{align}
 for all $x,y \in T$.     
\end{defn}

\begin{lem}
   Let $(T,[\cdot,\cdot,\cdot],\cdot)$ be a truss and  $N: T \to T$ be a Nijenhuis operator. Then the product 
   \begin{align}
     & x \circ y=  [N(x)\cdot y, N(x\cdot y), x\cdot N(y)]
   \end{align}
   provides a truss structure on $T$. Moreover, $N$ is a truss morphism from $(T,[\cdot,\cdot,\cdot],\circ)$ to $(T,[\cdot,\cdot,\cdot],\cdot)$. 
\end{lem}

\begin{prop}
        Let $(T,[\cdot,\cdot,\cdot],\cdot)$ be a truss and $R: T \to T$ be a Rota-Baxter operator of weight $\bf 0$. Define an operator on $T\times T$ by
        $$N(x,y)=(R(y),{\bf 0}),\quad\forall x,y\in T. $$
        Then $N$ is Nijenhuis operator on the truss $T\ltimes T$ given in Proposition \ref{TrussTimes}.
    \end{prop}
    \begin{proof}
        Since $R$ is a heap  homomorphism, then $N$ is a heap homomorphism on $T\times T$. Let $a,b,x,y\in T,$
        \begin{align*}
            N(a,x)\odot N(b,y)&=(R(x),{\bf 0})\odot (R(y),{\bf 0})\\
            &=(R(x)\cdot R(y),{\bf 0})\\
            &=(R[R(x)\cdot y,{\bf 0},x\cdot R(y)],{\bf 0}).
        \end{align*}
        On the other hand, we have 
        \begin{align*}
           &N( \Courant{N(a,x)\odot (b,y),N((a,x)\odot (b,y)),(a,x)\odot N(b,y)}) \\
           =&N( \Courant{(R(x),{\bf 0})\odot (b,y),N(a\cdot b,[a\cdot y, {\bf 0}, x\cdot b]),(a,x)\odot (R(y),{\bf 0})})
           \\
           =&N( \Courant{(R(x)\cdot b,R(x)\cdot y),(R([a\cdot y, {\bf 0}, x\cdot b]), {\bf 0}),(a\cdot R(y),x\cdot R(y))})\\
           =& \Courant{(R(R(x)\cdot y),{\bf 0}),({\bf 0}, {\bf 0}),(R(x\cdot R(y)),{\bf 0})}\\
            =&(R[R(x)\cdot y,{\bf 0},x\cdot R(y)],{\bf 0}).
        \end{align*}
        This completes the proof.
    \end{proof}

\begin{prop}
   Let $(T,[\cdot,\cdot,\cdot],\cdot)$ be a truss and  $N: T \to T$ be a Nijenhuis operator. Then the binary operations  
   \begin{align}
     & x \succ y= N(x)\cdot y, \quad
     x \ns y=N(x\cdot y), \quad  x \prec y=x\cdot N(y)
   \end{align}
   provide an NS-truss structure on $T$.    
\end{prop}

\begin{proof}
 Let $x,y,z \in T$. Then we have:
 \begin{align*}
 x \prec [y \succ z, y \ns z y \prec z]=& x \prec 
 [N(y)\cdot z, N(y \cdot z), y \cdot N(z)] \\
 =& x \cdot N([N(y)\cdot z, N(y \cdot z), y \cdot N(z)])\\
 =& x \cdot (N(y) \cdot N(z)) \\
 =& (x \cdot N(y)) \cdot N(z) \\
 =& (x \prec y) \prec z. 
 \end{align*}
 Hence, the identity \eqref{NS1} holds. 
 To show \eqref{NS2}, we compute as follows:
\begin{align*}
 (x \succ y) \prec z= & (N(x) \cdot y) \prec z \\
 =& (N(x) \cdot y) \cdot N(z) \\
 =& N(x) \cdot (y \cdot N(z) )= x \succ (y \prec z). 
\end{align*}
On the other hand, we have 
\begin{align*}
 [x \succ y,x \ns y,x \prec y]\ns z=& [N(x)\cdot y, N(x\cdot y),x \cdot N(y)] \ns z \\
 =& N([N(x)\cdot y, N(x\cdot y),x \cdot N(y)]) \cdot N(z)\\
 =& (N(x) \cdot N(y)) \cdot N(z) \\
 =& N(x) \cdot (N(y) \cdot N(z)) \\
 =& x \ns [y \succ z,y \ns z,y \prec z].
\end{align*}
Hence the identity \eqref{NS5} is proved. The other equalities can be checked in the same way. 
\end{proof}
We have shown that ring, trusses, dendriform rings,  dendriform trusses, NS-rings and NS-trusses are
closely related in the sense of commutative diagram of categories as follows:

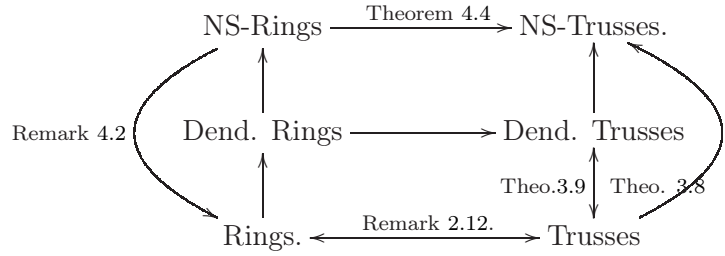
\begin{equation*} 
    \xymatrix{
\text{NS-Rings}
\ar[rr]^{\text{Theorem \ref{NSTrusToNSRing}}}\ar@/_4pc/[dd]_{\text{Remark \ref{NSRingToRing}}} & 
&\text{NS-Trusses.}  
  \\
\text{Dend. Rings}\ar[rr]^{ }  \ar[u]_{\text{ }}
&   & \text{Dend. Trusses} \ar[u]^{\text{ }} \ar[d]^{\text{ Theo. \ref{DendTrusToTrus}}}  \\
\text{Rings.} \ar[u]^{\text{ }}_{\text{ }}   \ar[rr]^{\text{Remark \ref{TrussesRings}.}} &    &
 \text{Trusses} \ar[u]^{\text{Theo.} \ref{TrusToDendTrus}} \ar@/_4pc/[uu]_{\text{ }} \ar[ll]  %\\
%\text{3-Lie-Rineh. alg.} \ar@<1ex>@{-->}[rr]^{\textit{  relative  RB. operator}} & &  \text{\it 3-pre-Lie-Rineh. alg.} \ar@<1ex>@{-->}[ll]^{\textit{  Commutator}} 
}
\end{equation*}

\section{Averaging operators on trusses, di-trusses and tri-trusses} \label{Sec5}
%-------------------------------------
In this section, we introduce the notion of (homomorphic) averaging operator on trusses. We investigate (homomorphic) averaging operators to construct di(tri)-trusses as a generalization of trusses. 

\begin{defn}\label{diass}
    A     di-truss $(T, [\cdot,\cdot,\cdot] ,\vdash,\dashv)$ consists of an abelian heap $(A,[\cdot,\cdot,\cdot])$ equipped with two associative   products $\vdash,\dashv:A\otimes A\xrightarrow{}A$ which are distributive with respect to bracket  and satisfy
\begin{align}
 &(x\dashv y)\vdash z = x\vdash (y\vdash z),\label{axiom3}\\&    (x\dashv y)\dashv z = x\dashv (y\vdash z),  \label{axiom4}\\
    &(x\vdash y)\dashv z = x\vdash (y\dashv z)  ,\label{axiom5}
\end{align}
for all $x,y,z\in A$.

\end{defn}
\begin{defn}
    A di-truss   $(A,[\cdot,\cdot,\cdot],\vdash,\dashv)$ is called tri-truss if it is equipped with associative product $\bot:A\otimes A\to A$ which are distributive with respect to bracket  and satisfy
    \begin{align}
      (x \dashv y) \dashv  z & =  x \dashv (y \,\bot\, z), \label{axiom6} \\
   (x \,\bot\, y) \dashv  z & =  x \,\bot\, (y \dashv z), \label{axiom7} \\
   (x \dashv y) \,\bot\,  z & =  x \,\bot\, (y \vdash z), \label{axiom8} \\
   (x \vdash y) \,\bot\,  z & =  x \vdash (y \,\bot\, z), \label{axiom9} \\
   (x \,\bot\, y) \vdash  z & =  x \vdash (y \vdash z), \label{axiom10}
   \end{align}
   for all $x,y,z\in A$.
\end{defn}
\begin{defn} \label{AverDef}
    Let $(T,[\cdot,\cdot,\cdot],\cdot)$ be a truss.  A map $\mathcal{K}:T\to T$ is called 
    \begin{enumerate}
        \item left
    averaging operator  if it is a heap  morphism of $(T,[\cdot,\cdot,\cdot])$ and satisfying 
    \begin{equation}
        \label{LeftAverCon}
        \mathcal{K}(x)\cdot \mathcal{K}(y)=\mathcal{K}(\mathcal{K}(x)\cdot y),\quad \forall x,y\in T.
    \end{equation}
    \item right
    averaging operator if it is a heap morphism of  $(T,[\cdot,\cdot,\cdot])$ and satisfying 
    \begin{equation}
        \label{RightAverCon}
        \mathcal{K}(x)\cdot \mathcal{K}(y)=\mathcal{K}(x\cdot \mathcal{K}(y)),\quad \forall x,y\in T.
    \end{equation}
    \item   averaging operator if it is both left and right averaging.
    \item homomorphic averaging operator if it is both a truss homomorphism and  averaging operator.
    \end{enumerate}
\end{defn}
In the following, we give a characterization of (homomorphic) averaging operators via graphs. Let $(T,[\cdot,\cdot,\cdot],\cdot)$ be a truss. Define on $T\times T$ the operations
\begin{align}
        &\Courant{(a,x),(b,y),(c,y)}=([a,b,c],[x,y,z]),\\
       & (a,x)\triangleright (b,y)=(a\cdot b,a\cdot y),\\
       & (a,x)\triangleleft (b,y)=(a\cdot b,x\cdot b),\\
       & (a,x)\triangle  (b,y)=(a\cdot b,x\cdot y),
    \end{align}
     for all $a,b,c,x,y,z\in T$. Then $(T\times T,\Courant{\cdot,\cdot,\cdot},\triangleright)$ (resp. $(T\times T,\Courant{\cdot,\cdot,\cdot},\triangleleft)$ and $(T\times T,\Courant{\cdot,\cdot,\cdot},\triangle)$)  is a truss which is  called left (resp. right and medium) hemisemi-direct product.   
     \begin{prop}
         A map $\mathcal{K}:T\to T$ is a left (resp. right) averaging operator if and only if its graph $${\sf Gr}(\mathcal{K}):=\{(\mathcal{K}(x),x),\  x\in T\}$$ is a sub-truss of  left  (resp. right) hemisemi-direct product on $T\times T$.
     \end{prop}
     \begin{proof}
        Straightforward. 
     \end{proof}
     \begin{coro}
        A map $\mathcal{K}:T\to T$ is a homomorphic averaging operator if and only if its graph ${\sf Gr}(\mathcal{K})$ is a sub-truss of  left , right and medium hemisemi-direct product on $T\times T$. 
     \end{coro}
\begin{prop}
    Let $\mathcal{K}$ be an  averaging operator on  a truss $(T,[\cdot,\cdot,\cdot])$. Then,  $(A,[\cdot,\cdot,\cdot],\vdash,\dashv)$ is a di-truss, where
    \begin{equation}
        \label{DiProduct}x\vdash y=\mathcal{K}(x)\cdot y \text{ and } x\dashv y=x\cdot \mathcal{K}(y),\quad \forall \ x,y\in T.
    \end{equation}
    Moreover, if $\mathcal{K}$ is homomorphic then $(A,[\cdot,\cdot,\cdot],\vdash,\dashv,\bot:=\cdot)$ is a tri-truss.  
\end{prop}
\begin{proof}
    Since $\mathcal K$ is a heap homomorphism, the the products $\vdash,\dashv$ are distributive with respect to the heap bracket. Let $x,y,z\in A$, by the associativity of "$\cdot$", we get
    \begin{align*}
        (x\dashv y)\vdash z &= \mathcal K(x\cdot \mathcal K(y))\cdot z\\&= (\mathcal K(x)\cdot \mathcal K(y))\cdot z\\&= \mathcal K(x)\cdot (\mathcal K(y)\cdot z)
        \\&=
        x\vdash (y\vdash z).
    \end{align*}
    \begin{align*}
        (x \dashv y) \dashv  z & = (x \cdot \mathcal K( y) )\cdot \mathcal K  (z)\\& = x \cdot (\mathcal K( y) \cdot \mathcal K  (z))\\& = x \cdot \mathcal K( y \cdot z)\\ 
        &=x \dashv (y \,\bot\, z).
    \end{align*}
    \begin{align*}
        (x \,\bot\, y) \dashv  z & =(x \,\cdot\, y) \cdot \mathcal K( z)\\& =x \,\cdot\,( y \cdot \mathcal K( z))\\
        &=
        x \,\bot\, (y \dashv z).
    \end{align*}
\end{proof}

%\begin{rmk}
%The notion of Hom-groups was initially introduced by Caenepeel and Goyvaerts in \cite{CaenepeelGoyvaerts} which    is a nonassociative version of a group where associativity, invertibility, and unitality are twisted by a map.      In the forthcoming papers, we are define ....\cite{}. 
%\end{rmk}

\end{document}